\documentclass[preprint]{elsarticle}
\usepackage{amsmath}
\usepackage{amsthm}
\usepackage{amssymb}
\usepackage{enumerate}
\usepackage[margin=1.00in]{geometry}
\usepackage{optidef}
\usepackage{lineno,hyperref}
\modulolinenumbers[5]

\newtheorem{theorem}{Theorem}[section]
\newtheorem{proposition}[theorem]{Proposition}
\newtheorem{corollary}[theorem]{Corollary}
\newtheorem{lemma}[theorem]{Lemma}

\theoremstyle{definition}
\newtheorem{definition}[theorem]{Definition}

\newcommand*\conj[1]{\overline{#1}}

\newcommand\abs[1]{\left|#1\right|}
\newcommand\norm[1]{\left\Vert#1\right\Vert}
\newcommand\fes[1]{\operatorname{Fes}\left(#1\right)}
\newcommand\opt[1]{\operatorname{Opt}\left(#1\right)}
\newcommand\dia[1]{\operatorname{d}\left(#1\right)}
\newcommand\lop[1]{\operatorname{LOP}\left(#1\right)}
\newcommand\tsp[1]{\operatorname{TSP}\left(#1\right)}

\newcommand\aff[1]{\operatorname{aff}\left(#1\right)}
\DeclareMathOperator*{\argmax}{argmax}

\DeclareMathOperator{\lcm}{lcm}
\DeclareMathOperator{\conv}{conv}
\DeclareMathOperator{\rank}{rank}

\begin{document}

\begin{frontmatter}
\title{Diameter Polytopes of Feasible Binary Programs}
\author{Thomas R. Cameron\fnref{pfootnote}}
\author{Sebastian Charmot\fnref{dfootnote}}
\author{Jonad Pulaj\fnref{dfootnote}}
\fntext[pfootnote]{Department of Mathematics, Penn State Erie the Behrend College, Erie, PA \,(trc5475@psu.edu)}
\fntext[dfootnote]{Department of Mathematics and Computer Science, Davidson College, Davidson, NC (secharmot@davidson.edu, jopulaj@davidson.edu)}

\begin{abstract}
Feasible binary programs often have multiple optimal solutions, which is of interest in applications as they allow the user to choose between alternative optima without deteriorating the objective function. 
In this article, we present the optimal diameter of a feasible binary program as a metric for measuring the diversity among all optimal solutions. 
In addition, we present the diameter binary program whose optima contains two optimal solutions of the given feasible binary program that are as diverse as possible with respect to the optimal diameter. 
Our primary interest is in the study of the diameter polytope, i.e., the polytope underlying the diameter binary program. 
Under suitable conditions, we show that much of the structure of the diameter polytope is inherited from the polytope underlying the given binary program. 
Finally, we apply our results on the diameter binary program and diameter polytope to cases where the given binary program corresponds to the linear ordering problem and the symmetric traveling salesman problem.
\end{abstract}

\begin{keyword}
 \texttt{linear ordering problem, maximum diversity, polyhedral theory, traveling salesman problem}
 \MSC[2010] 52B05 \sep 52B12 \sep 90C09 \sep 90C27 \sep 90C57
\end{keyword}

\end{frontmatter}
\section{Introduction}	
There is much interest in finding multiple optimal solutions for binary and integer programs, see~\cite{Petit2019,Tsai2008} and the references therein.
Of course, one can use integer cuts to remove previously found optimal solutions; however, many applications have too many optimal solutions for enumeration to be practical.
Therefore, it is reasonable to focus on multiple optimal solutions that are as diverse as possible~\cite{Glover1998,Kuo1993}, as uniform as possible~\cite{Kondo2014}, or are distinguishable by problem-specific parameters~\cite{Petit2019}.

In this article, we present the optimal diameter of a feasible binary program as a metric for measuring the diversity among all optimal solutions.
In addition, we present the diameter binary program whose optima contains two optimal solutions of the given feasible binary program that are as diverse as possible with respect to the optimal diameter. 
Our primary focus is the study of the diameter polytope, i.e., the polytope underlying the diameter binary program. 
In Section~\ref{subsec:dim-opt-dia-poly}, we show that under suitable conditions, the dimension of the diameter polytope can be obtained from the dimension of the polytope underlying the given binary program.
Moreover, in Section~\ref{subsec:dia-poly-facets}, under suitable conditions, we derive many facet inequalities for the diameter polytope, including facet inequalities inherited from the facets of the polytope underlying the given binary program.
Finally, we apply our results on the diameter binary program and diameter polytope to cases where the given binary program corresponds to the linear ordering problem (Section~\ref{sec:opt-dia-lop}) and the symmetric traveling salesman problem (Section~\ref{sec:opt-dia-tsp}).

Before proceeding, we note the similarities and differences between our work and several related prior works.
In all cases, the similarities only extend as far as the binary program models.
Indeed, theoretical investigation of the underlying polytopes is absent from the other works.

The diversity models in~\cite{Glover1998,Kuo1993} rely upon a given set of objects, whereas our model relies on a given binary program. 
In particular, in order to apply the models in~\cite{Glover1998,Kuo1993} to the optimal solutions of a binary program, one would first have to enumerate all optima.
As noted earlier, this is not practical in many applications. 

The model in~\cite{Kondo2014} is designed to find two optimal solutions to the linear ordering problem, over two objective functions, that are as uniform as possible. 
The diameter binary program (\eqref{eq:bpd-obj} --~\eqref{eq:bpd-bin-const}) can easily be adapted for such purposes. 
Indeed, the objective function in~\eqref{eq:bpd-obj} can be split as two objective functions over the variables $x$ and $y$.
Moreover, the conditions in~\eqref{eq:bpd-constz1} and~\eqref{eq:bpd-constz2} can be changed to seek two optima that are as uniform as possible.
\section{The Optimal Diameter of a Binary Program}\label{sec:opt-dia-bp}
Let $A\in\mathbb{R}^{n\times m}$, $b\in\mathbb{R}^{m}$, and $c\in\mathbb{R}^{n}$, and consider the general binary program, which we denote by $\textrm{BP}$:
\begin{maxi!}
	{}{c^{T}x}{}{}\label{eq:bp-obj}
	\addConstraint{Ax\leq b}\label{eq:bp-const}
	\addConstraint{x\in\{0,1\}^{n}}.
\end{maxi!}
The vector $x\in\{0,1,\}^{n}$ is a \emph{feasible solution} of $\textrm{BP}$ provided that $x$ satisfies~\eqref{eq:bp-const}.
If, in addition, $x$ is maximal with respect to the objective function~\eqref{eq:bp-obj}, then we say that $x$ is an \emph{optimal solution}.

Let $\fes{\textrm{BP}}$ and $\opt{\textrm{BP}}$ denote the set of feasible and optimal solutions, respectively, of a binary program.
Throughout this article, we assume that $\fes{\textrm{BP}}$ and, therefore, $\opt{\textrm{BP}}$ are non-empty.
In addition, we often denote feasible solutions by $\bar{x}$ and optimal solutions by $x^{*}$.
The following definition provides a metric for quantifying the pairwise diversity among the elements of $\opt{\textrm{BP}}$.
\begin{definition}
The \emph{optimal diameter} of a binary program is given by
\[
\dia{\textrm{BP}} := \argmax_{x^{*},y^{*}\in\opt{\textrm{BP}}}\norm{x^{*}-y^{*}}^{2},
\]
where $\norm{\cdot}$ denotes the Euclidean norm.
\end{definition}

Note that, since $x^{*},y^{*}\in\{0,1\}^{n}$, we can re-write the optimal diameter of $\textrm{BP}$ as
\[
\dia{\textrm{BP}} = \argmax_{x^{*},y^{*}\in\opt{\textrm{BP}}}\sum_{i=1}^{n}\abs{x^{*}_{i}-y^{*}_{i}},
\]
where $x^{*}_{i}$ denotes the $i$th entry of the vector $x^{*}$.
The following binary program, denoted by $\textrm{BPD}$, can be used to determine the optimal diameter of a given binary program:
\begin{maxi!}
	{}{c^{T}(x+y)-\epsilon e^{T}z}{}{}\label{eq:bpd-obj}
	\addConstraint{Ax\leq b}\label{eq:bpd-bp-constx}
	\addConstraint{Ay\leq b}\label{eq:bpd-bp-consty}
	\addConstraint{x+y-z\leq e}\label{eq:bpd-constz1}
	\addConstraint{-x-y-z}\leq -e\label{eq:bpd-constz2}
	\addConstraint{x,y,z\in\{0,1\}^{n}}\label{eq:bpd-bin-const},
\end{maxi!}
where $\epsilon>0$ and $e$ is the all ones vector of appropriate size. 
Since every binary program can be written in the canonical form of $\textrm{BP}$, analogous definitions for feasible and optimal solutions holds for $\textrm{BPD}$.
Throughout this article, we use $\conj{\textrm{BPD}}$ to denote $\textrm{BPD}$ without constraint~\eqref{eq:bpd-constz2}.
Also, let $[n]:=\{1,2,\ldots,n\}$.
Then, we have the following proposition regarding the optimal solutions of $\textrm{BPD}$ and $\conj{\textrm{BPD}}$.
\begin{proposition}\label{prop:bpd-obj}
Let $x^{*}\oplus y^{*}\oplus z^{*}\in\opt{\textrm{BPD}}$.
Then, for each $i\in [n]$, $z^{*}_{i}=1$ if and only if $x^{*}_{i}=y^{*}_{i}$.
Analogously, for each $x^{*}\oplus y^{*}\oplus z^{*}\in\opt{\conj{\textrm{BPD}}}$, $z^{*}_{i}=1$ if and only if $x^{*}_{i}=y^{*}_{i}=1$, for all $i\in [n]$.
\end{proposition}
\begin{proof}
Let $x^{*}\oplus y^{*}\oplus z^{*}\in\opt{\textrm{BPD}}$.
For the sake of contradiction, suppose that $z^{*}_{i}=1$ and $x^{*}_{i}\neq y^{*}_{i}$ for some $i\in [n]$.
Then, the constraints~\eqref{eq:bpd-constz1} and~\eqref{eq:bpd-constz2} are not satisfied with equality. 
Therefore, we can set $z^{*}_{i}=0$ and arrive at a feasible solution with a larger objective value in~\eqref{eq:bpd-obj}, thus contradicting the optimality assumption of $x^{*}\oplus y^{*}\oplus z^{*}$.
Conversely, suppose that $x^{*}\oplus y^{*}\oplus z^{*}\in\opt{\textrm{BPD}}$ and $x^{*}_{i}=y^{*}_{i}$ for some $i\in [n]$.
Then, in order for constraints~\eqref{eq:bpd-constz1} and~\eqref{eq:bpd-constz2} to be satisfied, it follows that $z^{*}_{i}=1$. 

A similar argument holds for $x^{*}\oplus y^{*}\oplus z^{*}\in\opt{\conj{\textrm{BPD}}}$.
\end{proof}

Note that $\epsilon>0$ is essential in the proof of Proposition~\ref{prop:bpd-obj} since it guarantees that setting $z^{*}_{i}=0$ will produce a feasible solution with a larger objective value in~\eqref{eq:bpd-obj}.
The following result shows that there exists an $\epsilon$ value such that the optimal diameter of the given $\textrm{BP}$ can be determined from any optimal solution of $\textrm{BPD}$.
In the proof, we make use of the following notation: $Z(x,y) := \left\{i\in[n]\colon x_{i}=y_{i}=0\right\}$, for all $x,y\in\{0,1\}^{n}$.
\begin{theorem}\label{thm:bpd-obj}
There exists an $\epsilon>0$ such that for all $x^{*}\oplus y^{*}\oplus z^{*}\in\opt{\textrm{BPD}}$, $\dia{\textrm{BP}} = n - e^{T}z^{*}$.
Analogously, there exists an $\epsilon>0$ such that for all $x^{*}\oplus y^{*}\oplus z^{*}\in\opt{\conj{\textrm{BPD}}}$, $\dia{\textrm{BP}} \leq n - e^{T}z^{*} $.
\end{theorem}
\begin{proof}
We break this proof into two cases: First, where $\fes{\textrm{BP}}=\opt{\textrm{BP}}$ and second, where $\fes{\textrm{BP}}\neq\opt{\textrm{BP}}$.
In the first case, it follows that there exists a $k\in\mathbb{R}$ such that $c^{T}(\bar{x}+\bar{y}) = k$, for all $\bar{x},\bar{y}\in\fes{\textrm{BP}}$.
Hence, for any $\epsilon>0$, maximizing the objective function in~\eqref{eq:bpd-obj} is equivalent to finding $x^{*},y^{*}\in\opt{BP}$ such that $e^{T}z^{*}$ is minimized.
By Proposition~\ref{prop:bpd-obj}, for any $x^{*}\oplus y^{*}\oplus z^{*}\in\opt{\textrm{BPD}}$, $n = e^{T}z^{*} + \norm{x^{*}-y^{*}}^{2}$, i.e.,
\begin{equation}\label{eq:thm-bpd-obj}
\norm{x^{*}-y^{*}}^{2} = n - e^{T}z^{*}.
\end{equation}
Since $e^{T}z^{*}$ is minimized, it follows that $\norm{x^{*}-y^{*}}^{2}$ is maximized and is therefore equal to $\dia{\textrm{BP}}$.

In the second case, there exists maximal $\bar{x}_{*},\bar{y}_{*}\in\fes{\textrm{BP}}\setminus{\opt{\textrm{BP}}}$ such that
\[
c^{T}(\bar{x}+\bar{y})\leq c^{T}(\bar{x}_{*}+\bar{y}_{*})<c^{T}(x^{*}+y^{*}),
\]
for all $\bar{x},\bar{y}\in\fes{\textrm{BP}}\setminus{\opt{\textrm{BP}}}$ and $x^{*},y^{*}\in\opt{\textrm{BP}}$.
Fix $x_{*},y_{*}\in\opt{\textrm{BP}}$ and set
\begin{equation}\label{eq:eps-bpd-obj}
\epsilon := \frac{c^{T}(x_{*}+y_{*}) - c^{T}(\bar{x}_{*}+\bar{y}_{*})}{2n}.
\end{equation}
Then, for any $\bar{x},\bar{y}\in\fes{\textrm{BP}}\setminus{\opt{\textrm{BP}}}$, it follows that
\[
c^{T}(\bar{x}+\bar{y})\leq c^{T}(x_{*}+y_{*}) - 2n\epsilon < c^{T}(x_{*}+y_{*}) - n\epsilon \leq c^{T}(x_{*}+y_{*}) - \epsilon e^{T}z,
\]
for any $z\in\{0,1\}^{n}$.
Hence, given any $x^{*}\oplus y^{*}\oplus z^{*}\in\opt{\textrm{BPD}}$, we have $x^{*},y^{*}\in\opt{\textrm{BP}}$.
Therefore, maximizing the objective function in~\eqref{eq:bpd-obj} is equivalent to finding $x^{*},y^{*}\in\opt{\textrm{BP}}$ such that $e^{T}z^{*}$ is minimized.
Again, by Proposition~\ref{prop:bpd-obj},~\eqref{eq:thm-bpd-obj} holds, where $\norm{x^{*}-y^{*}}^{2}=\dia{\textrm{BP}}$ since $e^{T}z^{*}$ is minimized.

A similar argument holds for $x^{*}\oplus y^{*}\oplus z^{*}\in\opt{\conj{\textrm{BPD}}}$, where Proposition~\ref{prop:bpd-obj} implies that
\[
\norm{x^{*}-y^{*}}^{2} = n - e^{T}z^{*} - \abs{Z(x^{*},y^{*})} \leq n - e^{T}z^{*}.
\]
Hence, the result follows from noting that the upper bound is maximized since $e^{T}z^{*}$ is minimized. 
\end{proof}

The value of $\epsilon$ in Theorem~\ref{thm:bpd-obj} is theoretical in nature as it relies on two optimal solutions of $\textrm{BP}$ and two maximal elements of $\fes{\textrm{BP}}\setminus{\opt{\textrm{BP}}}$. 
However, the following corollaries provide practical a priori values of $\epsilon$ that work under reasonable conditions.
\begin{corollary}\label{cor:int-bpd-obj}
Suppose that the vector $c$ in~\eqref{eq:bp-obj} and~\eqref{eq:bpd-obj} is integer valued and set $\epsilon := \frac{1}{2n}$.
Then, for any $x^{*}\oplus y^{*}\oplus z^{*}\in\opt{\textrm{BPD}}$, $\dia{\textrm{BP}} = n - e^{T}z^{*}$.
Analogously, for any $x^{*}\oplus y^{*}\oplus z^{*}\in\opt{\conj{\textrm{BPD}}}$, $\dia{\textrm{BP}} \leq n - e^{T}z^{*}$.
\end{corollary}
\begin{proof}
Fix $x_{*},y_{*}\in\opt{\textrm{BP}}$ and let $\bar{x}_{*},\bar{y}_{*}\in\fes{\textrm{{BP}}}\setminus{\opt{\textrm{BP}}}$ be maximal elements.
Then, we have $1\leq c^{T}(x_{*}+y_{*}) - c^{T}(\bar{x}_{*}+\bar{y}_{*})$, and it follows that
\[
\epsilon = \frac{1}{2n} \leq \frac{c^{T}(x_{*}+y_{*}) - c^{T}(\bar{x}_{*}+\bar{y}_{*})}{2n},
\]
where the rightmost fraction is equal to the value of epsilon in~\eqref{eq:eps-bpd-obj}.
Hence, the result follows from the proof of Theorem~\ref{thm:bpd-obj}.
\end{proof}
\begin{corollary}\label{cor:rat-bpd-obj}
Suppose that the vector $c$ in~\eqref{eq:bp-obj} and~\eqref{eq:bpd-obj} is rational valued, where $c=\left(\frac{a_{1}}{b_{1}},\ldots,\frac{a_{n}}{b_{n}}\right)$, and set $\epsilon := \frac{1}{2n\lcm\left(b_{1},\ldots,b_{n}\right)}$.
Then, for any $x^{*}\oplus y^{*}\oplus z^{*}\in\opt{\textrm{BPD}}$, $\dia{\textrm{BP}} = n - e^{T}z^{*}$.
Analogously, for any $x^{*}\oplus y^{*}\oplus z^{*}\in\opt{\conj{\textrm{BPD}}}$, $\dia{\textrm{BP}} \leq n - e^{T}z^{*}$.
\end{corollary}
\begin{proof}
Note that $\bar{c} = \lcm\left(b_{1},\ldots,b_{n}\right)\cdot c$ is an integer valued vector.
Hence, we can apply Corollary~\ref{cor:int-bpd-obj} to $\bar{c}$.
The result follows by dividing the corresponding objective function in~\eqref{eq:bpd-obj} by $\lcm\left(b_{1},\ldots,b_{n}\right)$.
\end{proof}

Note that Corollaries~\ref{cor:int-bpd-obj} and~\ref{cor:rat-bpd-obj} provide a practical method for computing $\dia{\textrm{BP}}$ by means of $\textrm{BPD}$ rather than computing the entire optimal set $\opt{\textrm{BP}}$. 
Furthermore, the following result shows that under reasonable conditions, an optimal solution of $\conj{\textrm{BPD}}$ can be used to compute $\dia{\textrm{BP}}$.
\begin{corollary}\label{cor:bar-bpd-obj}
Let $k$ be a non-negative integer such that $\norm{x^{*}}^{2} = k$ for all $x^{*}\in\opt{\textrm{BP}}$.
Then, there exists an $\epsilon>0$ such that for all $x^{*}\oplus y^{*}\oplus z^{*}\in\opt{\conj{\textrm{BPD}}}$, $\dia{\textrm{BP}} = 2\left(k - e^{T}z^{*}\right)$.
\end{corollary}
\begin{proof}
Let $x^{*}\oplus y^{*}\oplus z^{*}\in\opt{\conj{\textrm{BPD}}}$.
Then, by the proof of Theorem~\ref{thm:bpd-obj}, $x^{*},y^{*}\in\opt{\textrm{BP}}$ and $e^{T}z^{*}$ is minimized. 
Since $\norm{x^{*}}^{2} = k$ for all $x^{*}\in\opt{\textrm{BP}}$, we have
\[
\norm{x^{*}-y^{*}}^{2} = 2\left(k - e^{T}z^{*}\right).
\]
The result follows from noting that $\norm{x^{*}-y^{*}}^{2}=\dia{\textrm{BP}}$ since $e^{T}z^{*}$ is minimized. 
\end{proof}

The condition that $\norm{x^{*}}^{2}$ is constant over $\opt{\textrm{BP}}$ is satisfied by many important binary programs which correspond to well-known combinatorial optimization problems such as the linear ordering problem and the symmetric traveling salesman problem.
Therefore, we focus on the polytope underlying $\conj{\textrm{BPD}}$, which we reference as the \emph{diameter polytope} of $\textrm{BP}$ and define as follows:
\[
P_{\conj{\textrm{BPD}}}^{n} := \conv\left\{x\oplus y\oplus z\in\{0,1\}^{3n}\colon\text{constraints~\eqref{eq:bpd-bp-constx}--\eqref{eq:bpd-constz1} hold}\right\}.
\]
As we will see, there is much structure that this polytope inherits from the underlying polytope of $\textrm{BP}$, which we denote by $P^{n}_{\textrm{BP}}$.

\subsection{The Dimension of the Diameter Polytope}\label{subsec:dim-opt-dia-poly}
The \emph{dimension} of any polytope $P\subseteq\mathbb{R}^{n}$, denoted $\dim{P}$, is defined by the cardinality of the largest affinely independent subset of $P$~\cite{Grotschel1985:3}. 
Also, the dimension theorem states that $\dim{P}$ is equal to $n$ minus the maximum number of linearly independent equations satisfied by all points of $P$~\cite[Section 0.5]{Lee2004}.

The \emph{minimal equation system} of a polytope $P\subseteq\mathbb{R}^{n}$, $Mx=d$, where $M\in\mathbb{R}^{m\times n}$ and $d\in\mathbb{R}^{m}$, is the largest possible collection of linearly independent equations satisfied by all points of $P$. 
If the polytope $P\subseteq\mathbb{R}^{n}$ is \emph{full dimensional}, i.e., $\dim{P} = n$, then no such minimal equation system exists since there is no hyperplane containing $P$. 
Otherwise, the dimension theorem implies that $\dim{P} = n - \rank{M}$.

Under suitable conditions on the $\textrm{BP}$, we can use the above observations to determine the dimension of the diameter polytope $\textrm{BP}$.
The following theorem establishes these conditions and their effect on the feasible solutions of $\conj{\textrm{BPD}}$.
\begin{theorem}\label{thm:dia-poly-dim}
Suppose that there exists $\bar{x},\bar{y}\in\fes{\textrm{BP}}$ such that $\bar{x}+\bar{y}\leq e$.
Let $d\in\mathbb{R}^{3n}$ and $d_{0}\in\mathbb{R}$ such that $d^{T}\left(\bar{x}\oplus\bar{y}\oplus\bar{z}\right)=d_{0}$ for all $\bar{x}\oplus\bar{y}\oplus\bar{z}\in\fes{\conj{\textrm{BPD}}}$.
If we decompose $d = d_{x}\oplus d_{y}\oplus d_{z}$, where $d_{x},d_{y},d_{z}\in\mathbb{R}^{n}$, then $d_{z}=0$.
Furthermore, there exists $c_{x},c_{y}\in\mathbb{R}$ such that $d_{x}^{T}\bar{x}=c_{x}$ and $d_{y}^{T}\bar{y}=c_{y}$ for all $\bar{x},\bar{y}\in\fes{\textrm{BP}}$. 
\end{theorem}
\begin{proof}
Let $\bar{x},\bar{y}\in\fes{\textrm{BP}}$ such that $\bar{x}+\bar{y}\leq e$.
Then, $\bar{x}\oplus\bar{y}\oplus\bar{z}\in\fes{\conj{\textrm{BPD}}}$ for all $\bar{z}\in\{0,1\}^{n}$.
Fix $i\in[n]$ and set $\bar{z}\in\{0,1\}^{n}$ such that $\bar{z}_{i}=1$ and all other entries are zero. 
Then, $d^{T}\left(\bar{x}\oplus\bar{y}\oplus 0\right) = d^{T}\left(\bar{x}\oplus\bar{y}\oplus\bar{z}\right)$ implies that $d_{z_{i}}=0$.
Since $i\in[n]$ is arbitrary, it follows that $d_{z}=0$. 

Now, let $\bar{x},\bar{y}\in\fes{\textrm{BP}}$ and $\bar{z}\in\{0,1\}^{n}$ such that $\bar{x}\oplus\bar{y}\oplus\bar{z}\in\fes{\conj{\textrm{BPD}}}$.
Then, $\bar{y}\oplus\bar{x}\oplus\bar{z}\in\fes{\conj{\textrm{BPD}}}$, and it follows that $d^{T}\left(\bar{x}\oplus\bar{y}\oplus\bar{z}\right) = d^{T}\left(\bar{y}\oplus\bar{x}\oplus\bar{z}\right)$.
Therefore,
\begin{equation}\label{eq:thm-dia-poly-dim1}
d^{T}_{x}\bar{x} - d^{T}_{y}\bar{x} = d^{T}_{x}\bar{y} - d^{T}_{y}\bar{y}.
\end{equation}
Similarly, $\bar{x}\oplus\bar{x}\oplus\bar{x}\in\fes{\conj{\textrm{BPD}}}$ and $\bar{y}\oplus\bar{y}\oplus\bar{y}\in\fes{\conj{\textrm{BPD}}}$ implies that
\begin{equation}\label{eq:thm-dia-poly-dim2}
d^{T}_{x}\bar{x} + d^{T}_{y}\bar{x} = d^{T}_{x}\bar{y} + d^{T}_{y}\bar{y}.
\end{equation}
Adding~\eqref{eq:thm-dia-poly-dim1} and~\eqref{eq:thm-dia-poly-dim2} gives us
\begin{equation}\label{eq:thm-dia-poly-dim3}
d^{T}_{x}\bar{x} = d^{T}_{x}\bar{y},
\end{equation}
for all $\bar{x},\bar{y}\in\fes{\textrm{BP}}$.

Hence, if we temporarily fix $\bar{y}\in\fes{\textrm{BP}}$ and set $c_{x} := d^{T}_{x}\bar{y}$, then~\eqref{eq:thm-dia-poly-dim3} implies that $d^{T}_{x}\bar{x}=c_{x}$ for all $\bar{x}\in\fes{\textrm{BP}}$.
Moreover, subtracting~\eqref{eq:thm-dia-poly-dim1} from~\eqref{eq:thm-dia-poly-dim2} gives us
\begin{equation}\label{eq:thm-dia-poly-dim4}
d^{T}_{y}\bar{x} = d^{T}_{y}\bar{y},
\end{equation}
for all $\bar{x},\bar{y}\in\fes{\textrm{BP}}$.
Again, if we temporarily fix $\bar{x}\in\fes{\textrm{BP}}$ and set $c_{y} := d^{T}_{y}\bar{x}$, then~\eqref{eq:thm-dia-poly-dim4} implies that $d^{T}_{y}\bar{y} = c_{y}$ for all $\bar{y}\in\fes{\textrm{BP}}$.
\end{proof}

The following corollaries use Theorem~\ref{thm:dia-poly-dim} to establish the dimension of $P^{n}_{\conj{\textrm{BPD}}}$, both when the polytope is full dimensional and when it is not full dimensional.

\begin{corollary}\label{cor:dia-poly-dim1}
Suppose that there exists $\bar{x},\bar{y}\in\fes{\textrm{BP}}$ such that $\bar{x}+\bar{y}\leq e$.
If $P^{n}_{\textrm{BP}}$ is full dimensional, then $P^{n}_{\conj{\textrm{BPD}}}$ is full dimensional, i.e.,
\[
\dim{P^{n}_{\conj{\textrm{BPD}}}} = 3n.
\]
\end{corollary}
\begin{proof}
Let $d\in\mathbb{R}^{3n}$ and $d_{0}\in\mathbb{R}$ such that $d^{T}\left(\bar{x}\oplus\bar{y}\oplus\bar{z}\right)=d_{0}$ for all $\bar{x}\oplus\bar{y}\oplus\bar{z}\in\fes{\conj{\textrm{BPD}}}$.
Decompose $d=d_{x}\oplus d_{y}\oplus d_{z}$, where $d_{x},d_{y},d_{z}\in\mathbb{R}^{n}$.
Then, by Theorem~\ref{thm:dia-poly-dim}, $d_{z}=0$.
Furthermore, there exists $c_{x},c_{y}\in\mathbb{R}$ such that 
\[
d^{T}_{x}\bar{x} = c_{x}~\text{and}~d^{T}_{y}\bar{y} = c_{y},
\]
for all $\bar{x},\bar{y}\in\fes{\textrm{BP}}$. 
Since $P^{n}_{\textrm{BP}}$ is full dimensional, it follows that $d_{x}=0$ and $d_{y}=0$.
Therefore, $d=0$ and it follows that $P^{n}_{\conj{\textrm{BPD}}}$ is full dimensional. 
\end{proof}
\begin{corollary}\label{cor:dia-poly-dim2}
Suppose that there exists $\bar{x},\bar{y}\in\fes{\textrm{BP}}$ such that $\bar{x}+\bar{y}\leq e$.
Also, suppose that $Mx=d$, where $M\in\mathbb{R}^{m\times n}$ and $d\in\mathbb{R}^{m}$, is a minimal equation of $P^{n}_{\textrm{BP}}$.
Let $O_{2m\times n}$ be the $2m\times n$ zero matrix, and define $\hat{M} := \left[M\oplus M~O_{2m\times n}\right]$ and $\hat{d} = d\oplus d$.
Then, $\hat{M}\left(x\oplus y\oplus z\right)=\hat{d}$ is a minimal equation system for $P^{n}_{\conj{\textrm{BPD}}}$ and, hence,
\[
\dim{P^{n}_{\conj{\textrm{BPD}}}} = 3n - 2\rank{M}.
\]
\end{corollary}
\begin{proof}
Let $d\in\mathbb{R}^{3n}$ and $d_{0}\in\mathbb{R}$ such that $d^{T}\left(\bar{x}\oplus\bar{y}\oplus\bar{z}\right)=d_{0}$ for all $\bar{x}\oplus\bar{y}\oplus\bar{z}\in\fes{\conj{\textrm{BPD}}}$.
Decompose $d=d_{x}\oplus d_{y}\oplus d_{z}$, where $d_{x},d_{y},d_{z}\in\mathbb{R}^{n}$.
Then, by Theorem~\ref{thm:dia-poly-dim}, $d_{z}=0$.
Furthermore, there exists $c_{x},c_{y}\in\mathbb{R}$ such that 
\[
d^{T}_{x}\bar{x} = c_{x}~\text{and}~d^{T}_{y}\bar{y} = c_{y},
\]
for all $\bar{x},\bar{y}\in\fes{\textrm{BP}}$. 
It follows that $d^{T}_{x}x = c_{x}$ and $d^{T}_{y}y = c_{y}$ must be linear combinations of the minimal equation system for $P^{n}_{\textrm{BP}}$. 
Therefore, $d^{T}\left(x\oplus y\oplus z\right) = d_{0}$ can be written as a linear combination of the equation system $\hat{M}\left(x\oplus y\oplus z\right) = \hat{d}$. 
Since $M$ has full rank, we know that $\hat{M}$ has full rank and, hence, $\hat{M}\left(x\oplus y\oplus z\right) = \hat{d}$ is a minimal equation system of $P^{n}_{\conj{\textrm{BPD}}}$.
The result follows from noting that $\rank{\hat{M}} = 2\rank{M}$.
\end{proof}

\subsection{Facets of the Diameter Polytope}\label{subsec:dia-poly-facets}
Let $P\subseteq\mathbb{R}^{n}$ be a polytope and let $a^{T}x\leq a_{0}$ denote a valid inequality of $P$.
A \emph{face} $F\subseteq P$ is defined by $F = \left\{x\in P\colon a^{T}x = a_{0}\right\}$.
We say that $F$ is a face of $P$ defined by the inequality $a^{T}x\leq a_{0}$.
A facet of $P$ is a face of $P$ whose dimension is equal to $\dim{P} - 1$. 

In what follows, we show that under suitable conditions of $\textrm{BP}$, we can establish several facets of $P^{n}_{\conj{\textrm{BPD}}}$. 
To this end, we will use the indirect method as described in~\cite[Theorem 1]{Grotschel1985:3}.
For reference, we summarize this method in the theorem below.
Note that $\aff{P}$ denotes the affine hull of the polytope $P$.
\begin{theorem}\label{thm:poly-facet-ind-method}
Let $P\subseteq\mathbb{R}^{n}$ be a polytope and assume that $A\in\mathbb{R}^{m\times n}$ and $b\in\mathbb{R}^{m}$ satisfy $\aff{P} = \left\{x\in\mathbb{R}^{n}\colon Ax=b\right\}$.
Let $F$ be a face of $P$ defined by the inequality $a^{T}x\leq a_{0}$. 
Then, $F$ is a facet if and only if the following hold:
\begin{enumerate}[(a)]
\item	There exists an $\tilde{x}\in P$ such that $a^{T}\tilde{x}<a_{0}$. 
\item	If any other valid inequality $d^{T}x\leq d_{0}$ of $P$ satisfies $F \subseteq\left\{x\in P\colon d^{T}x = d_{0}\right\}$,
then there exists a scalar $\alpha\geq 0$ and a vector $\lambda\in\mathbb{R}^{m}$ such that
\begin{align*}
d^{T} &= \alpha a^{T} + \lambda^{T}A, \\
d_{0} &= \alpha a_{0} + \lambda^{T}b.
\end{align*}
\end{enumerate}
\end{theorem}

The following result shows that under suitable conditions of $\textrm{BP}$, we can use the facets of $P^{n}_{\textrm{BP}}$ to determine facets of $P^{n}_{\conj{\textrm{BPD}}}$. 
Note that these conditions are stronger than those needed in Theorem~\ref{thm:dia-poly-dim} and Corollaries~\ref{cor:dia-poly-dim1} and~\ref{cor:dia-poly-dim2}.
\begin{theorem}\label{thm:dia-poly-facet}
Suppose that for each $\bar{x}\in\fes{\textrm{BP}}$, there exists $\bar{y}\in\fes{\textrm{BP}}$ such that $\bar{x}+\bar{y}\leq e$.
Let $a^{T}x\leq a_{0}$ define a facet of $P^{n}_{\textrm{BP}}$.
Then, $\hat{a}^{T}\left(x\oplus y\oplus z\right)\leq a_{0}$ is a facet defining inequality of $P^{n}_{\conj{\textrm{BPD}}}$ for $\hat{a}=a\oplus 0\oplus 0$ and $\hat{a}=0\oplus a\oplus 0$.
\end{theorem}
\begin{proof}
Let $\hat{a} := a\oplus 0\oplus 0$. 
In what follows, we show that $\hat{a}^{T}\left(x\oplus y\oplus z\right)\leq a_{0}$ is a facet defining inequality of $P^{n}_{\conj{\textrm{BPD}}}$.
A similar approach can be used for $\hat{a} := 0\oplus a\oplus 0$.

Let $F := \left\{\bar{x}\in\fes{\textrm{BP}}\colon a^{T}\bar{x} = a_{0}\right\}$ denote a facet of $P^{n}_{\textrm{BP}}$ and define a face of $P^{n}_{\conj{\textrm{BPD}}}$ as follows:
\[
\hat{F} := \left\{\bar{x}\oplus\bar{y}\oplus\bar{z}\in\fes{\conj{\textrm{BPD}}}\colon \hat{a}^{T}\left(\bar{x}\oplus\bar{y}\oplus\bar{z}\right) = a_{0}\right\}.
\]
Also, let $A\in\mathbb{R}^{m\times n}$ and $b\in\mathbb{R}^{m}$ such that $\aff{P^{n}_{\textrm{BP}}} = \left\{x\in\mathbb{R}^{n}\colon Ax=b\right\}$.
Note that if $P^{n}_{\textrm{BP}}$ is full dimensional, then $A$ and $b$ can be taken to be all zero. 
Otherwise, $A$ and $b$ can be formed from the minimal equation of $P^{n}_{\textrm{BP}}$.
Define $\hat{A} := \left[A\oplus A~O_{2m\times n}\right]$, where $O_{2m\times n}$ is the $2m\times n$ zero matrix, and $\hat{b} = b\oplus b$. 
It follows from Corollary~\ref{cor:dia-poly-dim1} and~\ref{cor:dia-poly-dim2} that 
\[
\aff{P^{n}_{\conj{\textrm{BPD}}}} = \left\{\bar{x}\oplus\bar{y}\oplus\bar{z}\in\mathbb{R}^{3n}\colon \hat{A}\left(\bar{x}\oplus\bar{y}\oplus\bar{z}\right) = \hat{b}\right\}.
\]

Note that, by Theorem~\ref{thm:poly-facet-ind-method}(a), there exists a vector $\tilde{x}\in\fes{\textrm{BP}}$ such that $a^{T}\tilde{x}<a_{0}$. 
Therefore, $\tilde{x}\oplus\tilde{x}\oplus\tilde{x}\in\fes{\conj{\textrm{BPD}}}$ satisfies $\hat{a}^{T}\left(\tilde{x}\oplus\tilde{x}\oplus\tilde{x}\right) < a_{0}$, and it follows that Theorem~\ref{thm:poly-facet-ind-method}(a) holds for the face $\hat{F}$ of $P^{n}_{\conj{\textrm{BPD}}}$. 

Suppose that there exists a valid inequality $d^{T}\left(x\oplus y\oplus z\right)\leq d_{0}$ of $P^{n}_{\conj{\textrm{BPD}}}$ such that
\[
\hat{F}\subseteq\left\{\bar{x}\oplus\bar{y}\oplus\bar{z}\in\fes{\conj{\textrm{BPD}}}\colon d^{T}\left(\bar{x}\oplus\bar{y}\oplus\bar{z}\right)=d_{0}\right\}.
\]
Let $\bar{x}\in F$.
By hypothesis, there exists a $\bar{y}\in\fes{\textrm{BP}}$ such that $\bar{x}+\bar{y}\leq e$. 
Hence, $\bar{x}\oplus\bar{y}\oplus\bar{z}\in \hat{F}$ for all $\bar{z}\in\{0,1\}^{n}$.
Fix $i\in[n]$ and decompose $d = d_{x}\oplus d_{y}\oplus d_{z}$, where $d_{x},d_{y},d_{z}\in\mathbb{R}^{n}$.
Define $\bar{z}\in\{0,1\}^{n}$ by $\bar{z}_{i}=1$ and all other entries are zero.
Then, $d^{T}\left(\bar{x}\oplus\bar{y}\oplus 0\right) = d^{T}\left(\bar{x}\oplus\bar{y}\oplus\bar{z}\right)$ implies that $d_{z_{i}}=0$.
Since $i\in[n]$ is arbitrary, it follows that $d_{z} = 0$. 

Now, let $\bar{x},\bar{y}\in F$ and $\bar{z}\in\{0,1\}^{n}$ such that $\bar{x}\oplus\bar{y}\oplus\bar{z}\in\fes{\conj{\textrm{BPD}}}$. 
Then, it is clear that both $\bar{x}\oplus\bar{y}\oplus\bar{z}\in\hat{F}$ and $\bar{y}\oplus\bar{x}\oplus\bar{z}\in\hat{F}$.
Therefore,
\begin{equation}\label{eq:thm-dia-poly-facet1}
d^{T}_{x}\bar{x} - d^{T}_{y}\bar{x} = d^{T}_{x}\bar{y} - d^{T}_{y}\bar{y}.
\end{equation}
Similarly, $\bar{x}\oplus\bar{x}\oplus\bar{x}\in\hat{F}$ and $\bar{y}\oplus\bar{y}\oplus\bar{y}\in\hat{F}$, which implies that
\begin{equation}\label{eq:thm-dia-poly-facet2}
d^{T}_{x}\bar{x} + d^{T}_{y}\bar{x} = d^{T}_{x}\bar{y} + d^{T}_{y}\bar{y}.
\end{equation}
Adding~\eqref{eq:thm-dia-poly-facet1} and~\eqref{eq:thm-dia-poly-facet2} gives us
\begin{equation}\label{eq:thm-dia-poly-facet3}
d^{T}_{x}\bar{x} = d^{T}_{x}\bar{y},
\end{equation}
for all $\bar{x},\bar{y}\in F$. 
Hence, if we temporarily fix $\bar{y}\in F$ and define $c_{x} := d^{T}_{x}\bar{y}$, then~\eqref{eq:thm-dia-poly-facet3} implies that $d^{T}_{x}\bar{x} = c_{x}$ for all $\bar{x}\in F$.

Next, consider $\bar{x}\in F$.
Then, for all $\bar{y}\in\fes{\textrm{BP}}$, we have $\bar{x}\oplus \bar{y}\oplus \bar{y}\in\hat{F}$.
Note that $d^{T}_{x}\bar{x} + d^{T}_{y}\bar{y} = d_{0}$ implies that $d^{T}_{y}\bar{y} = d_{0} - c_{x} := c_{y}$, for all $\bar{y}\in\fes{\textrm{BP}}$.
Furthermore, $d^{T}_{x}x\leq c_{x}$ defines a valid inequality for $P^{n}_{\textrm{BP}}$.
Otherwise, there exists an $\bar{x}\in\fes{\textrm{BP}}$ such that $d^{T}_{x}\bar{x}>c_{x}$, i.e., for any $\bar{y}\in\fes{\textrm{BP}}$, we have $\bar{x}\oplus\bar{y}\oplus\bar{y}\in\fes{\conj{\textrm{BPD}}}$ and
\[
d^{T}\left(\bar{x}\oplus\bar{y}\oplus\bar{y}\right) = d^{T}_{x}\bar{x} + d^{T}_{y}\bar{y} > d_{0},
\]
which contradicts $d^{T}\left(x\oplus y\oplus z\right)\leq d_{0}$ defining a valid inequality of $P^{n}_{\conj{\textrm{BPD}}}$. 

Since $d^{T}_{y}\bar{y} = c_{y}$ for all $\bar{y}\in\fes{\textrm{BP}}$, there exists a vector $\lambda_{y}\in\mathbb{R}^{m}$ such that
\begin{equation}\label{eq:thm-dia-poly-facet4}
d^{T}_{y} = \lambda_{y}^{T}A.
\end{equation}
Also, since $d^{T}_{x}x\leq c_{x}$ defines a valid inequality of $P^{n}_{\textrm{BP}}$ such that $F\subseteq\left\{x\in\fes{\textrm{BP}}\colon d^{T}_{x}x = c_{x}\right\}$,
it follows from Theorem~\ref{thm:poly-facet-ind-method}(b) that there exists a scalar $\alpha_{x}\geq 0$ and vector $\lambda_{x}\in\mathbb{R}^{m}$ such that
\begin{equation}\label{eq:thm-dia-poly-facet5}
\begin{split}
d^{T}_{x} &= \alpha_{x}a^{T} + \lambda_{x}^{T}A, \\
c_{x} &= \alpha_{x}a_{0} + \lambda_{x}^{T}b.
\end{split}
\end{equation}
Combining~\eqref{eq:thm-dia-poly-facet4} and~\eqref{eq:thm-dia-poly-facet5}, gives us
\begin{align*}
d^{T} &= \alpha_{x}\hat{a}^{T} + \lambda^{T}\hat{A}, \\
d_{0} &= \alpha_{x}a_{0} + \lambda^{T}\hat{b},
\end{align*}
where $\lambda = \lambda_{x}\oplus \lambda_{y}$. 
Therefore, Theorem~\ref{thm:poly-facet-ind-method}(b) holds for the face $\hat{F}$ of $P^{n}_{\conj{\textrm{BPD}}}$, and the result follows. 
\end{proof}

Next, we show that under the same conditions of Theorem~\ref{thm:dia-poly-facet} , the trivial hypercube constraints $0\leq z_{i}\leq 1$ define facets of $P^{n}_{\conj{\textrm{BPD}}}$. 
\begin{theorem}\label{thm:dia-poly-zfacet}
Suppose that for each $\bar{x}\in\fes{\textrm{BP}}$ there exists $\bar{y}\in\fes{\textrm{BP}}$ such that $\bar{x}+\bar{y}\leq e$.
Then $z_{i}\geq 0$ and $z_{i}\leq 1$ are facet defining inequalities of $P^{n}_{\conj{\textrm{BPD}}}$, for all $i\in[n]$.
\end{theorem}
\begin{proof}
In what follows we show that $z_{i}\leq 1$ is a facet defining inequality of $P^{n}_{\conj{\textrm{BPD}}}$, for all $i\in [n]$.
Note that a similar argument can be made for the inequality $z_{i}\geq 0$.

Fix $i\in[n[$ and define a face of $P^{n}_{\conj{\textrm{BPD}}}$ as follows:
\[
F := \left\{\bar{x}\oplus\bar{y}\oplus\bar{z}\in\fes{\conj{\textrm{BPD}}}\colon \bar{z}_{i}=1\right\}.
\]
Also, let $A$, $b$, $\hat{A}$, $\hat{b}$ be defined as in Theorem~\ref{thm:dia-poly-facet}.
Then, it follows from Corollary~\ref{cor:dia-poly-dim1} and~\ref{cor:dia-poly-dim2} that 
\[
\aff{P^{n}_{\conj{\textrm{BPD}}}} = \left\{\bar{x}\oplus\bar{y}\oplus\bar{z}\in\mathbb{R}^{3n}\colon \hat{A}\left(\bar{x}\oplus\bar{y}\oplus\bar{z}\right) = \hat{b}\right\}.
\]

Let $\bar{x}\in\fes{\textrm{BP}}$.
By hypothesis, there exists a $\bar{y}\in\fes{\textrm{BP}}$ such that $\bar{x}+\bar{y}\leq e$.
Therefore, $\bar{x}\oplus\bar{y}\oplus\bar{z}\in\fes{\conj{\textrm{BPD}}}$ for all $\bar{z}\in\{0,1\}^{n}$.
Setting $\bar{z}_{i}=0$ implies that Theorem~\ref{thm:poly-facet-ind-method}(a) holds for the face $F$ of $P^{n}_{\conj{\textrm{BPD}}}$.
Now, suppose that there exists a valid inequality $d^{T}\left(x\oplus y\oplus z\right)\leq d_{0}$ of $P^{n}_{\conj{\textrm{BPD}}}$ such that 
\[
F\subseteq\left\{\bar{x}\oplus\bar{y}\oplus\bar{z}\in\fes{\conj{\textrm{BPD}}}\colon d^{T}\left(\bar{x}\oplus\bar{y}\oplus\bar{z}\right) = d_{0}\right\}
\]
and note that $\bar{x}\oplus\bar{y}\oplus\bar{z}\in F$ for all $\bar{z}\in\{0,1\}^{n}$ such that $\bar{z}_{i}=1$.
Fix $j\in[n]\setminus\{i\}$ and decompose $d=d_{x}\oplus d_{y}\oplus d_{z}$, where $d_{x},d_{y},d_{z}\in\mathbb{R}^{n}$.
Define $\bar{z}\in\{0,1\}^{n}$ by $\bar{z}_{i}=1$ and all other entries zero; also, define $\hat{z}\in\{0,1\}^{n}$ by $\hat{z}_{i}=\hat{z}_{j}=1$ and all other entries zero.
Then, $d^{T}\left(\bar{x}\oplus\bar{y}\oplus\bar{z}\right) = d^{T}\left(\bar{x}\oplus\bar{y}\oplus\hat{z}\right)$ implies that $d_{z_{j}}=0$.
Since $j\in[n]\setminus\{i\}$ is arbitrary, it follows that $d_{z_{j}}=0$ for all $j\in[n]\setminus\{i\}$. 

Now, let $\bar{x},\bar{y}\in\fes{\textrm{BP}}$ and $\bar{z}\in\{0,1\}^{n}$ such that $\bar{x}\oplus\bar{y}\oplus\bar{z}\in F$.
Then, 
\[
d^{T}\left(\bar{x}\oplus\bar{y}\oplus\bar{z}\right) = d^{T}_{x}\bar{x} + d^{T}_{y}\bar{y} + d_{z_{i}}.
\]
Moreover, $\bar{y}\oplus\bar{x}\oplus\bar{z}\in F$, and it follows that
\begin{equation}\label{eq:thm-dia-poly-zfacet1}
d^{T}_{x}\bar{x} - d^{T}_{y}\bar{x} = d^{T}_{x}\bar{y} - d^{T}_{y}\bar{y}.
\end{equation}
Similarly, there exists $\bar{z},\hat{z}\in\{0,1\}^{n}$ such that $\bar{x}\oplus\bar{x}\oplus\bar{z}\in F$ and $\bar{y}\oplus\bar{y}\oplus\hat{z}\in F$, which implies that
\begin{equation}\label{eq:thm-dia-poly-zfacet2}
d^{T}_{x}\bar{x} + d^{T}_{y}\bar{x} = d^{T}_{x}\bar{y} + d^{T}_{y}\bar{y}.
\end{equation}
Adding~\eqref{eq:thm-dia-poly-zfacet1} and~\eqref{eq:thm-dia-poly-zfacet2} gives us
\begin{equation}\label{eq:thm-dia-poly-zfacet3}
d^{T}_{x}\bar{x} = d^{T}_{x}\bar{y},
\end{equation}
for all $\bar{x},\bar{y}\in\fes{\textrm{BP}}$.

Hence, if we temporarily fix $\bar{y}\in\fes{\textrm{BP}}$ and define $c_{x} := d^{T}_{x}\hat{y}$, then~\eqref{eq:thm-dia-poly-zfacet3} implies that $d^{T}_{x}\bar{x} = c_{x}$ for all $\bar{x}\in\fes{\textrm{BP}}$. 
Moreover, subtracting~\eqref{eq:thm-dia-poly-zfacet1} from~\eqref{eq:thm-dia-poly-zfacet2} gives us
\begin{equation}\label{eq:thm-dia-poly-zfacet4}
d^{T}_{y}\bar{x} = d^{T}_{y}\bar{y},
\end{equation}
for all $\bar{x},\bar{y}\in\fes{\textrm{BP}}$.
Again, if we temporarily fix $\bar{x}\in\fes{\textrm{BP}}$ and set $c_{y} := d^{T}_{y}\bar{x}$, then~\eqref{eq:thm-dia-poly-zfacet4} implies that $d^{T}_{y} = c_{y}$ for all $\bar{y}\in\fes{\textrm{BP}}$.

Since $d^{T}_{x}\bar{x} = c_{x}$ and $d^{T}_{y}\bar{y} = c_{y}$ for all $\bar{x},\bar{y}\in\fes{\textrm{BP}}$, there exists vectors $\lambda_{x},\lambda_{y}\in\mathbb{R}^{m}$ such that 
\begin{equation}\label{eq:thm-dia-poly-zfacet5}
\begin{split}
d^{T}_{x} &= \lambda^{T}_{x}A, \\
d^{T}_{y} &= \lambda^{T}_{y}A.
\end{split}
\end{equation}
Therefore, we have
\begin{align*}
d^{T} &= d_{z_{i}}e^{T}_{i} + \lambda^{T}\hat{A}, \\
d_{0} &= d_{z_{i}} + \lambda^{T}\hat{b},
\end{align*}
where $\lambda = \lambda_{x}\oplus\lambda_{y}$ and $e_{i}$ denotes the $i$th standard basis vector of $\mathbb{R}^{3n}$.
Hence, Theorem~\ref{thm:poly-facet-ind-method}(b) holds for the face $F$ of $P^{n}_{\conj{\textrm{BPD}}}$, and the result follows. 
\end{proof}

Finally, we show that under the same conditions of Theorem~\ref{thm:dia-poly-facet}, the inequality in~\eqref{eq:bpd-constz1} defines a facet of $P^{n}_{\conj{\textrm{BPD}}}$.
Without loss of generality, we assume that for each $i\in [n]$ there exists an $x\in\fes{\textrm{BP}}$ such that $x_{i}=1$; otherwise, we can project to a lower dimensional space and consider the polytope in $\mathbb{R}^{n-1}$.
\begin{theorem}\label{thm:dia-poly-ktfacet}
Suppose that for each $\bar{x}\in\fes{\textrm{BP}}$ there exists a $\bar{y}\in\fes{\textrm{BP}}$ such that $\bar{x}+\bar{y}\leq e$.
Then, $x_{i}+y_{i}-z_{i}\leq 1$ defines a facet of $P^{n}_{\conj{\textrm{BPD}}}$,  for all $i\in[n]$.
\end{theorem}
\begin{proof}
Fix $i\in[n]$ and define a face of $P^{n}_{\conj{\textrm{BPD}}}$ as follows:
\[
F := \left\{\bar{x}\oplus\bar{y}\oplus\bar{z}\in\fes{\conj{\textrm{BPD}}}\colon x_{i}+y_{i}-z_{i}=1\right\}.
\]
Also, let $A$, $b$, $\hat{A}$, $\hat{b}$ be defined as in Theorem~\ref{thm:dia-poly-facet}.
Then, it follows from Corollary~\ref{cor:dia-poly-dim1} and~\ref{cor:dia-poly-dim2} that 
\[
\aff{P^{n}_{\conj{\textrm{BPD}}}} = \left\{\bar{x}\oplus\bar{y}\oplus\bar{z}\in\mathbb{R}^{3n}\colon \hat{A}\left(\bar{x}\oplus\bar{y}\oplus\bar{z}\right) = \hat{b}\right\}.
\]

Let $\bar{x}\in\fes{\textrm{BP}}$ such that $\bar{x}_{i}=1$.
By hypothesis, there exists a $\bar{y}\in\fes{\textrm{BP}}$ such that $\bar{x}+\bar{y}\leq e$, which implies that $\bar{x}\oplus\bar{y}\oplus\bar{z}\in\fes{\conj{\textrm{BPD}}}$ for all $\bar{z}\in\{0,1\}^{n}$. 
If we select $\bar{z}_{i}=1$, then it is clear that $\bar{x}_{i}+\bar{y}_{i}-\bar{z}_{i}<1$; hence, Theorem~\ref{thm:poly-facet-ind-method}(a) holds for the face $F$ of $P^{n}_{\conj{\textrm{BPD}}}$. 
Now, suppose that there exists a valid inequality $d^{T}\left(x\oplus y\oplus z\right)\leq d_{0}$ of $P^{n}_{\conj{\textrm{BPD}}}$ such that 
\[
F\subseteq\left\{\bar{x}\oplus\bar{y}\oplus\bar{z}\in\fes{\conj{\textrm{BPD}}}\colon d^{T}\left(\bar{x}\oplus\bar{y}\oplus\bar{z}\right)=d_{0}\right\},
\]
and note that $\bar{x}\oplus\bar{y}\oplus\bar{z}\in F$ for all $\bar{z}\in\{0,1\}^{n}$ such that $\bar{z}_{i}=0$.
Fix $j\in[n]\setminus\{i\}$ and decompose $d = d_{x}\oplus d_{y}\oplus d_{z}$.
Define $\bar{z}\in\{0,1\}$ by $\bar{z}_{j}=1$ and all other entries equal to zero. 
Then, $d^{T}\left(\bar{x}\oplus\bar{y}\oplus 0\right) = d^{T}\left(\bar{x}\oplus\bar{y}\oplus\bar{z}\right)$ implies that $d_{z_{j}}=0$.
Since $j\in[n]\setminus\{i\}$ is arbitrary, it follows that $d_{z_{j}}=0$ for all $j\in[n]\setminus\{i\}$.

Now, let $\bar{x},\bar{y}\in\fes{\textrm{BP}}$ and $\bar{z}\in\{0,1\}^{n}$, such that $\bar{x}_{i}=\bar{y}_{i}=1$ and $\bar{x}\oplus\bar{y}\oplus\bar{z}\in F$.
Then,
\[
d^{T}\left(\bar{x}\oplus\bar{y}\oplus\bar{z}\right) = d^{T}_{x}\bar{x}+d^{T}_{y}\bar{y} + d_{z_{i}}. 
\]
Moreover, $\bar{y}\oplus\bar{x}\oplus\bar{z}\in F$ and it follows that
\begin{equation}\label{eq:thm-dia-poly-ktfacet1}
d^{T}_{x}\bar{x} - d^{T}_{y}\bar{x} = d^{T}_{x}\bar{y} - d^{T}_{y}\bar{y}.
\end{equation}
Similarly, there exists $\bar{z},\hat{z}\in\{0,1\}^{n}$ such that $\bar{x}\oplus\bar{x}\oplus\bar{z}\in F$ and $\bar{y}\oplus\bar{y}\oplus\hat{z}\in F$, which implies that
\begin{equation}\label{eq:thm-dia-poly-ktfacet2}
d^{T}_{x}\bar{x} + d^{T}_{y}\bar{x} = d^{T}_{x}\bar{y} + d^{T}_{y}\bar{y}.
\end{equation}
Adding~\eqref{eq:thm-dia-poly-ktfacet1} and~\eqref{eq:thm-dia-poly-ktfacet2} gives us
\begin{equation}\label{eq:thm-dia-poly-ktfacet3}
d^{T}_{x}\bar{x} = d^{T}_{x}\bar{y},
\end{equation}
for all $\bar{x},\bar{y}\in\fes{\textrm{BP}}$ such that $\bar{x}_{i}=\bar{y}_{i}=1$. 

Hence, if we temporarily fix $\bar{y}\in\fes{\textrm{BP}}$ such that $\bar{y}_{i}=1$ and define $c_{x} := d^{T}_{x}\bar{y}$, then~\eqref{eq:thm-dia-poly-ktfacet3} implies that $d^{T}_{x}\bar{x} = c_{x}$ for all $\bar{x}\in\fes{\textrm{BP}}$ such that $\bar{x}_{i}=1$. 
Moreover, subtracting~\eqref{eq:thm-dia-poly-ktfacet1} from~\eqref{eq:thm-dia-poly-ktfacet2} gives us
\begin{equation}\label{eq:thm-dia-poly-ktfacet4}
d^{T}_{y}\bar{x} = d^{T}_{y}\bar{y},
\end{equation}
for all $\bar{x},\bar{y}\in\fes{\textrm{BP}}$ such that $\bar{x}_{i}=\bar{y}_{i}=1$.
Again, if we temporarily fix $\bar{x}\in\fes{\textrm{BP}}$ such that $\bar{x}_{i}=1$ and set $c_{y} := d^{T}_{y}\bar{x}$, then~\eqref{eq:thm-dia-poly-ktfacet4} implies that $d^{T}_{y}\bar{y} = c_{y}$ for all $\bar{y}\in\fes{\textrm{BP}}$ such that $\bar{y}_{i}=1$. 

Finally, fix $\bar{x},\bar{y}\in\fes{\textrm{BP}}$, where $\bar{x}_{i}=\bar{y}_{i}=1$, and $\bar{z}\in\{0,1\}^{n}$ such that $\bar{x}\oplus\bar{y}\oplus\bar{z}\in F$.
For any $\hat{y}\in\fes{\textrm{BP}}$, where $\hat{y}_{i}=0$, there exists a $\hat{z}\in\{0,1\}^{n}$ such that $\bar{x}\oplus\hat{y}\oplus\hat{z}\in F$.
Hence, we have $d^{T}_{y}\bar{y} = d^{T}_{y}\hat{y}$, and it follows that $d^{T}_{y}\hat{y} = c_{y}$ for all $\hat{y}\in\fes{\textrm{BP}}$ such that $\hat{y}_{i}=0$.
Similarly, for any $\hat{x}\in\fes{\textrm{BP}}$, where $\hat{x}_{i}=0$, there exists a $\hat{z}\in\{0,1\}^{n}$ such that $\hat{x}\oplus\bar{y}\oplus\hat{z}\in F$.
Hence, we have $d^{T}_{x}\bar{x} = d^{T}_{x}\hat{x}$, and it follows that $d^{T}_{x}\hat{x} = c_{x}$ for all $\hat{x}\in\fes{\textrm{BP}}$ such that $\hat{x}_{i}=0$. 

Therefore, $d^{T}_{x}\bar{x} = c_{x}$ and $d^{T}_{y}\bar{y} = c_{y}$ for all $\bar{x},\bar{y}\in\fes{\textrm{BP}}$, and it follows that there exists vectors $\lambda_{x},\lambda_{y}\in\mathbb{R}^{m}$ such that 
\begin{align*}
d^{T}_{x} &= \lambda^{T}_{x}A, \\
d^{T}_{y} &= \lambda^{T}_{y}A.
\end{align*}
Hence, we have
\begin{align*}
d^{T} &= d_{z_{i}}e^{T}_{i} + \lambda^{T}\hat{A}, \\
d_{0} &= d_{z_{i}} + \lambda^{T}\hat{b},
\end{align*}
where $\lambda = \lambda_{x}\oplus\lambda_{y}$ and $e_{i}$ denotes the $i$th standard basis vector of $\mathbb{R}^{3n}$.
Thus, Theorem~\ref{thm:poly-facet-ind-method}(b) holds for the face $F$ of $P^{n}_{\conj{\textrm{BPD}}}$ and the result follows. 
\end{proof}

\section{The Optimal Diameter of the LOP}\label{sec:opt-dia-lop}
Let $n\geq 2$ and $a\in\mathbb{R}^{n(n-1)}$ be a vector with entries $a_{ij}$, where $i\neq j$ and $i,j\in[n]$.
The \emph{linear ordering problem}, denoted $\lop{a}$, is defined as follows~\cite{Marti2011}:
\begin{maxi!}
	{}{\sum_{i\neq j\colon i,j\in[n]}a_{ij}x_{ij}}{}{}\label{eq:lop-obj}
	\addConstraint{x_{ij}+x_{ji}}{=1,\quad\forall i<j\colon i,j\in[n]}\label{eq:lop-sym-const}
	\addConstraint{x_{ij}+x_{jk}+x_{ki}}{\leq 2,\quad\forall i<j,i<k,j\neq k\colon i,j,k\in[n]}\label{eq:lop-trans-const}
	\addConstraint{x_{ij}}{\in\{0,1\},\quad\forall i\neq j\colon i,j\in[n]}.
\end{maxi!}
Let $S_{n}$ denote the set of permutations on $[n]$.
It is well-known that every feasible solution of $\lop{a}$ corresponds to a unique permutation $\sigma\in S_{n}$, where $\bar{x}_{ij} = 1$ if and only if $\sigma(i)<\sigma(j)$, for all $i,j\in[n]$ such that $i\neq j$.
We say that $\sigma\in S_{n}$ is an \emph{optimal permutation} if it corresponds to an optimal solution of $\lop{a}$.

For two permutations $\sigma_{1},\sigma_{2}\in S_{n}$, we define the set of all \emph{discordant pairs} as
\[
D(\sigma_{1},\sigma_{2}) := \left\{(i,j)\in[n]\times[n]\colon i<j\land\left(\sigma_{1}(i)<\sigma_{1}(j)\land\sigma_{2}(i)>\sigma_{2}(j)\lor\sigma_{1}(i)>\sigma_{1}(j)\land\sigma_{2}(i)<\sigma_{2}(j)\right)\right\}
\]
and the set of \emph{concordant pairs} as
\[
C(\sigma_{1},\sigma_{2}) := T - D(\sigma_{1},\sigma_{2}),
\]
where $T := \left\{(i,j)\in[n]\times[n]\colon i<j\right\}$.
Then, the Kendall tau distance is defined as follows. 
\begin{definition}
The \emph{Kendall tau} distance between $\sigma_{1},\sigma_{2}\in S_{n}$ is given by
\[
K(\sigma_{1},\sigma_{2}) = \abs{D(\sigma_{1},\sigma_{2})}.
\]
\end{definition}

For $\epsilon>0$, the optimal diameter binary program for $\lop{a}$, denoted $\conj{\textrm{LOD}}$, is defined by:
\begin{maxi!}
	{}{\sum_{i\neq j\colon i,j\in[n]}a_{ij}\left(x_{ij}+y_{ij}\right) - \epsilon z_{ij}}{}{}\label{eq:dlop-obj}
	\addConstraint{x_{ij}+x_{ji}}{=1,\quad\forall i<j\colon i,j\in[n]}\label{eq:dlop-sym-constx}
	\addConstraint{y_{ij}+y_{ji}}{=1,\quad\forall i<j\colon i,j\in[n]}\label{eq:dlop-sym-consty}
	\addConstraint{x_{ij}+x_{jk}+x_{ki}}{\leq 2,\quad\forall i<j,i<k,j\neq k\colon i,j,k\in[n]}\label{eq:dlop-trans-constx}
	\addConstraint{y_{ij}+y_{jk}+y_{ki}}{\leq 2,\quad\forall i<j,i<k,j\neq k\colon i,j,k\in[n]}\label{eq:dlop-trans-consty}
	\addConstraint{x_{ij}+y_{ij}-z_{ij}}{\leq 1,\quad\forall i\neq j\colon i,j,k\in[n]}\label{eq:dlop-constz}
	\addConstraint{x_{ij},y_{ij},z_{ij}\in\{0,1\},\quad\forall i\neq j\colon i,j\in[n]}.
\end{maxi!}
The following result shows that the optimal diameter of $\lop{a}$ is proportional to the maximal Kendall tau distance between all optimal permutations.
\begin{proposition}\label{prop:lop-obj}
Let $S$ denote the set of optimal permutations for $\lop{a}$.
Then, the optimal diameter of $\lop{a}$ satisfies
\[
\dia{\lop{a}} = 2\cdot\argmax_{\sigma_{1},\sigma_{2}\in S}K(\sigma_{1},\sigma_{2}).
\]
\end{proposition}
\begin{proof}
Note that $\norm{\bar{x}}^{2}=\binom{n}{2}$ for all $\bar{x}\in\fes{\textrm{LOP}}$.
By Corollary~\ref{cor:bar-bpd-obj}, there exists an $\epsilon>0$ such that every $x^{*}\oplus y^{*}\oplus z^{*}\in\opt{\conj{\textrm{LOD}}}$ satisfies 
\[
\dia{\lop{a}} = 2\left(\binom{n}{2} - e^{T}z^{*}\right).
\]
Recall from the proof of Theorem~\ref{thm:bpd-obj} that $x^{*},y^{*}\in\opt{\lop{a}}$ such that $e^{T}z^{*}$ is minimized.

Now, let $\sigma_{1},\sigma_{2}\in S_{n}$ be optimal rankings of $\lop{a}$ corresponding to $x^{*}$ and $y^{*}$, respectively.
Then, by Proposition~\ref{prop:bpd-obj}, $z^{*}_{ij}=1$ if and only if $\sigma_{1}(i)<\sigma_{1}(j)$ and $\sigma_{2}(i)<\sigma_{2}(j)$, where $(i,j)\in C(\sigma_{1},\sigma_{2})$ if $i<j$, and $(j,i)\in C(\sigma_{1},\sigma_{2})$ if $j<i$.
Therefore, $e^{T}z^{*}=\abs{C(\sigma_{1},\sigma_{2})}$, and we have
\[
\binom{n}{2} - e^{T}z^{*} = \abs{D(\sigma_{1},\sigma_{2})} = K(\sigma_{1},\sigma_{2}).
\]
The result follows from noting that $K(\sigma_{1},\sigma_{2})$ is maximized since $e^{T}z^{*}$ is minimized. 
\end{proof}

From the proof of Proposition~\ref{prop:lop-obj}, there exists an $\epsilon>0$ such that every $x^{*}\oplus y^{*}\oplus z^{*}\in\opt{\conj{\textrm{LOD}}}$ satisfies
\begin{equation}\label{eq:lop-kt-obj}
\argmax_{\sigma_{1},\sigma_{2}\in S}K(\sigma_{1},\sigma_{2}) = \binom{n}{2} - e^{T}z^{*},
\end{equation}
where $S$ is the set of all optimal permutations for $\lop{a}$.
Furthermore, by Corollary~\ref{cor:int-bpd-obj} and~\ref{cor:rat-bpd-obj}, if the entries $a_{ij}$, for $i,j\in[n]$ such that $i\neq j$, are integer or rational valued, then there is a readily computable value of $\epsilon$ that will suffice for~\eqref{eq:lop-kt-obj} to hold.

Now, we turn our attention to the diameter polytope of the LOP, which is defined as follows
\[
P^{n}_{\conj{\textrm{LOD}}} :=  \conv\left\{x\oplus y\oplus z\in\{0,1\}^{3n(n-1)}\colon\text{constraints~\eqref{eq:dlop-sym-constx}--\eqref{eq:dlop-constz} hold}\right\}.
\]
Let $\bar{x}\in\fes{\lop{a}}$ and $\sigma_{1}\in S_{n}$ denote the corresponding permutation, where $\bar{x}_{ij}=1$ if and only if $\sigma_{1}(i)<\sigma_{1}(j)$, for all $i,j\in[n]$ such that $i\neq j$. 
Then, it is clear that there exists a $\bar{y}\in\fes{\lop{a}}$ and corresponding permutation $\sigma_{2}\in S{n}$, such that $\bar{y}_{ij}=1$ if and only if $\bar{x}_{ij}=0$, i.e., $\sigma_{2}(i)<\sigma_{2}(j)$ if and only if $\sigma_{1}(i)>\sigma_{1}(j)$. 
Therefore, the conditions in Corollary~\ref{cor:dia-poly-dim2} apply to $\lop{a}$, and we have the following result.
\begin{proposition}\label{prop:lop-dia-poly-dim}
The dimension of $P^{n}_{\conj{\textrm{LOD}}}$ satisfies
\[
\dim{P^{n}_{\conj{\textrm{LOD}}}} = 2n(n-1). 
\]
\end{proposition}
\begin{proof}
Let $Mx=d$, where $M\in\mathbb{R}^{m\times n(n-1)}$ and $d\in\mathbb{R}^{m}$, be a minimal equation of $P^{n}_{\textrm{LO}}$, i.e., the linear ordering polytope. 
Then, Corollary~\ref{cor:dia-poly-dim2} implies that
\[
\dim{P^{n}_{\conj{\textrm{LOD}}}} = 3n(n-1) - 2\rank{M},
\]
where $\rank{M} = \binom{n}{2}$ by~\cite[Theorem 2.5]{Grotschel1985:2}.
\end{proof}

In fact, the conditions in Theorems~\ref{thm:dia-poly-facet}--\ref{thm:dia-poly-ktfacet} apply to $\lop{a}$, and the result below follows immediately.
In particular, note that Proposition~\ref{prop:lop-dia-poly-facets}(a) can be used in conjunction with~\cite[Theorem 3.14]{Grotschel1985:2} to produce multiple facet inequalities for $P^{n}_{\conj{\textrm{LOD}}}$. 
\begin{proposition}\label{prop:lop-dia-poly-facets}
Let $a^{T}x\leq a_{0}$ define a facet of $P^{n}_{\textrm{LO}}$.
Then, the following are facet inequalities of $P^{n}_{\conj{\textrm{LOD}}}$:
\begin{enumerate}[(a)]
\item $\hat{a}^{T}\left(x\oplus y\oplus z\right)\leq a_{0}$, for $\hat{a}=a\oplus 0\oplus 0$ and $\hat{a}=0\oplus a\oplus 0$,
\item $z_{ij}\geq 0$ and $z_{ij}\leq 1$, for all $i,j\in[n]$ such that $i\neq j$,
\item $x_{ij}+y_{ij}-z_{ij}\leq 1$, for all $i,j\in[n]$ such that $i\neq j$. 
\end{enumerate}
\end{proposition} 

Finally, we prove a lifting result analogous to~\cite[Lemma 3.1]{Grotschel1985:2}.
To this end, for each $\sigma\in S_{n}$, define the corresponding \emph{incidence vector} $\chi^{\sigma}\in\{0,1\}^{n(n-1)}$ by $\chi_{ij} = 1$ if and only if $\sigma(i)<\sigma(j)$, for all $i,j\in[n]$ such that $i\neq j$.
For each $\sigma_{1},\sigma_{2}\in S_{n}$, define the triple $T := (\sigma_{1},\sigma_{2},Z)$, where $Z\in\{0,1\}^{n(n-1)}$ is selected so that the \emph{(partial) incidence vector} $\chi^{T} := \chi^{\sigma_{1}}\oplus\chi^{\sigma_{2}}\oplus Z$ is an element of $\fes{\conj{\textrm{LOD}}}$.
\begin{theorem}\label{thm:lop-dia-poly-facet-lift}
Suppose $a^{T}\left(x\oplus y\oplus z\right)\leq a_{0}$ defines a facet of $P^{n}_{\conj{\textrm{LOD}}}$, where $a=a_{x}\oplus a_{y}\oplus a_{z}$ and $a_{x},a_{y},a_{z}\in\mathbb{R}^{n(n-1)}$.
Define $\hat{a}_{x_{ij}} := a_{x_{ij}}$, for all $i,j\in[n]$ such that $i\neq j$, and $\hat{a}_{x_{i,n+1}} := \hat{a}_{x_{n+1,i}} := 0$ for all $i\in[n]$. 
Also, let $\hat{a}_{y}$ and $\hat{a}_{z}$ be defined similarly.
Then, $\hat{a}\left(x\oplus y\oplus z\right)\leq a_{0}$ defines a facet of $P^{n+1}_{\conj{\textrm{LOD}}}$, where $\hat{a} = \hat{a}_{x}\oplus\hat{a}_{y}\oplus\hat{a}_{z}$ and $\hat{a}_{x},\hat{a}_{y},\hat{a}_{z}\in\mathbb{R}^{n(n+1)}$.
\end{theorem}
\begin{proof}
In contrast to the results in Section~\ref{subsec:dia-poly-facets}, here we will use the direct method from~\cite[Theorem 2]{Grotschel1985:3}, i.e., we will construct $\dim{P^{n+1}_{\conj{\textrm{LOD}}}}$ linearly independent vectors $\hat{x}\oplus\hat{y}\oplus\hat{z}\in P^{n+1}_{\conj{\textrm{LOD}}}$ that satisfy $\hat{a}\left(\hat{x}\oplus\hat{y}\oplus\hat{z}\right)\leq a_{0}$ with equality. 
We note that in this case, affine and linear independence are equivalent since the zero vector is not contained in the affine hull of $P^{n+1}_{\conj{\textrm{LOD}}}$. 

Given our hypothesis, we can find $d := 2n(n-1)$ triples $T_{k} := \left(\sigma^{k}_{1},\sigma^{k}_{2},Z^{k}\right)$, for $1\leq k\leq d$, such that their (partial) incidence vectors $\chi^{T_{1}},\chi^{T_{2}},\ldots,\chi^{T_{d}}$ are linearly independent and satisfy $a^{T}\left(x\oplus y\oplus z\right)\leq a_{0}$ with equality. 
Let $M'$ denote a $d\times 3n(n-1)$ matrix whose $k$th row corresponds to the vector $\chi^{T_{k}}$. 
Since $M'$ is full rank, it has a $d\times d$ non-singular submatrix, which we denote by $M$. 
We now construct a larger matrix $N'$ whose rows are linearly independent vectors in $P^{n+1}_{\conj{\textrm{LOD}}}$ that satisfy $\hat{a}^{T}\left(\hat{x}\oplus\hat{y}\oplus\hat{z}\right)\leq a_{0}$ with equality. 
Note that $N'$ will have $d_{1} := 2n(n+1) = d+4n$ rows and $3n(n+1)$ columns; hence, it will suffice to exhibit a $d_{1}\times d_{1}$ non-singular submatrix $N$ of $N'$ by selecting certain columns of $N'$. 

For $1\leq k\leq d$, construct a new triple $\hat{T}_{k} := \left(\hat{\sigma}^{k}_{1},\hat{\sigma}^{k}_{2},\hat{Z}^{k}\right)$ by setting $\hat{\sigma}^{k}_{1}(n+1) := 1$, $\hat{\sigma}^{k}_{1}(i) := \sigma^{k}_{1}(i) + 1$, for $i\in[n]$, and $\hat{\sigma}^{k}_{2}(n+1) := n+1$, $\hat{\sigma}^{k}_{2}(i) = \sigma^{k}_{2}(i)$, for $i\in [n]$, and $\hat{Z}^{k}_{ij} = 1$ if and only if $\hat{\sigma}^{k}_{1}(i)<\hat{\sigma}^{k}_{1}(j)$ and $\hat{\sigma}^{k}_{2}(i)<\hat{\sigma}^{k}_{2}(j)$, for $i,j\in[n+1]$ such that $i\neq j$. 
The corresponding (partial) incidence vectors will form the first block of $d$ rows of $N'$.

Next, let $T = \left(\sigma_{1},\sigma_{2},Z\right)$ be any triple whose partial incidence vector satisfies $a^{T}\left(x\oplus y\oplus z\right)\leq a_{0}$ with equality. 
For $1\leq k\leq n$, construct a new triple $\hat{S}_{k} := \left(\hat{\sigma}^{k}_{1},\hat{\sigma}^{k}_{2},\hat{Z}^{k}\right)$ by setting $\hat{\sigma}^{k}_{1}(n+1) := 1$, $\hat{\sigma}^{k}_{1}(i) := \sigma_{1}(i) + 1$, for $i\in[n]$, and $\hat{\sigma}^{k}_{2}(n+1) := n+1$, $\hat{\sigma}^{k}_{2}(i) = \sigma_{2}(i)$, for $i\in [n]$, and $\hat{Z}^{k}_{ij} = 1$ if and only if $j=n+1$ and $i=k$ or $\hat{\sigma}^{k}_{1}(i)<\hat{\sigma}^{k}_{1}(j)$ and $\hat{\sigma}^{k}_{2}(i)<\hat{\sigma}^{k}_{2}(j)$, for $i,j\in[n+1]$.
The corresponding partial incidence vectors will form the second block of $n$ rows of $N'$.
Similarly, for $1\leq k\leq n$, construct a new triple $\bar{S}_{k} := \left(\bar{\sigma}^{k}_{1},\bar{\sigma}^{k}_{2},\bar{Z}^{k}\right)$ by setting $\bar{\sigma}^{k}_{1}(n+1) := 1$, $\bar{\sigma}^{k}_{1}(i) := \sigma_{1}(i) + 1$, for $i\in[n]$, and $\bar{\sigma}^{k}_{2}(n+1) := n+1$, $\bar{\sigma}^{k}_{2}(i) = \sigma_{2}(i)$, for $i\in [n]$, and $\bar{Z}^{k}_{ij} = 1$ if and only if $i=n+1$ and $j=k$ or $\bar{\sigma}^{k}_{1}(i)<\bar{\sigma}^{k}_{1}(j)$ and $\bar{\sigma}^{k}_{2}(i)<\bar{\sigma}^{k}_{2}(j)$, for $i,j\in[n+1]$.
The corresponding partial incidence vectors will form the third block of $n$ rows of $N'$. 

Finally, let $T = \left(\sigma_{1},\sigma_{2},Z\right)$ be any triple whose partial incidence vector satisfies $a^{T}\left(x\oplus y\oplus z\right)\leq a_{0}$ with equality. 
For $1\leq k\leq n$, construct a new triple $\hat{R}_{k} := \left(\hat{\sigma}^{k}_{1},\hat{\sigma}^{k}_{2},\hat{Z}^{k}\right)$ by setting $\hat{\sigma}^{k}_{1}(n+1) := k+1$, and $\hat{\sigma}^{k}_{1}(i) := \sigma_{1}(i)$ if $\sigma_{1}(i)<k+1$, otherwise, $\hat{\sigma}^{k}_{1}(i) := \sigma_{1}(i) + 1$, for $i\in [n]$. 
Furthermore, set $\hat{\sigma}^{k}_{2}(n+1) := n+1$, $\hat{\sigma}^{k}_{2}(i) := \sigma_{2}(i)$, for $i\in[n]$, and $\hat{Z}^{k}_{ij} := 1$ if and only if $\hat{\sigma}^{k}_{1}(i)<\hat{\sigma}^{k}_{1}(j)$ and $\hat{\sigma}^{k}_{2}(i)<\hat{\sigma}^{k}_{2}(j)$, for $i,j\in[n+1]$ such that $i\neq j$. 
The corresponding partial incidence vectors will form the fourth block of $n$ rows of $N'$.
Similarly, for $1\leq k\leq n$, construct a new triple $\bar{R}_{k} := \left(\bar{\sigma}^{k}_{1},\bar{\sigma}^{k}_{2},\bar{Z}^{k}\right)$ by setting $\bar{\sigma}^{k}_{2}(n+1) := k+1$, and $\bar{\sigma}^{k}_{2}(i) := \sigma_{2}(i)$ if $\sigma_{2}(i)<k+1$, otherwise, $\bar{\sigma}^{k}_{2}(i) := \sigma_{2}(i) + 1$, for $i\in [n]$. 
Furthermore, set $\bar{\sigma}^{k}_{1}(n+1) := 1$, $\bar{\sigma}^{k}_{1}(i) := \sigma_{1}(i)+1$, for $i\in[n]$, and $\bar{Z}^{k}_{ij} := 1$ if and only if $\bar{\sigma}^{k}_{1}(i)<\bar{\sigma}^{k}_{1}(j)$ and $\bar{\sigma}^{k}_{2}(i)<\bar{\sigma}^{k}_{2}(j)$, for $i,j\in[n+1]$ such that $i\neq j$. 
The corresponding partial incidence vectors will form the fifth block of $n$ rows of $N'$.

Hence, we have constructed all $d_{1}$ rows of $N'$.
All that remains is to select $d_{1}$ columns of $N'$ to form a non-singular submatrix $N$. 
To that end, the first block of $d$ columns correspond to the columns of $M$, the second block of $n$ columns correspond to $\chi^{\sigma_{1}}_{i,n+1}$ for $i\in[n]$, the third block of $n$ columns correspond to $\chi^{\sigma_{2}}_{n+1,j}$ for $j\in[n]$, the fourth block of $n$ columns corresponds to $Z_{i,n+1}$ for $i\in[n]$, and the fifth block of $n$ columns corresponds to $Z_{n+1,j}$ for $j\in[n]$. 
It is easy to verify that $N$ can be put in the following form:
\[
N = \begin{bmatrix} M & 0 & 0 & 0 & 0 \\ * & 0 & 0 & I & 0 \\ * & 0 & 0 & 0 & I \\ * & R & 0 & * & * \\ * & 0 & R & * & * \end{bmatrix},
\]
where $*$ denotes entries that are irrelevant to the non-singularity of $N$, $R\in\{0,1\}^{n\times n}$ such that $R_{ij} = 1$ if and only if $i\leq j$, and $0$ and $I$ denote all zero and identity matrices of appropriate sizes, respectively.
Indeed, note that $\abs{\det{N}} = \abs{\det{M}} > 0$ and the result follows.
\end{proof}
\section{The Optimal Diameter of the Symmetric TSP}\label{sec:opt-dia-tsp}
The Traveling Salesman Problem (TSP) is the prototype of combinatorial optimization problems where advances in the theory of polyhedral combinatorics have led to spectacular computational results~\cite{Grotschel1993}.
Before defining the symmetric traveling salesman problem, we introduce the pertinent graph theory definitions.

A \emph{graph} is a pair $G = (V,E)$, where $V$ is a non-empty finite set and $E$ is a set of two element subsets of $V$. 
For convenience, we denote the edge $\{i,j\}\in E$ by $ij$. 
We say that $H$ is a \emph{subgraph} of $G$ provided that $V(H)\subseteq V(G)$ and $E(H)\subseteq E(G)$; the subgraph is denoted by $H\subseteq G$. 
In particular, $H$ is an \emph{induced subgraph} of $G$ if there exists an $A\subseteq V(G)$ such that $V(H) = A$ and $E(H) = \left\{ij\in E(G)\colon i,j\in A\right\}$; the induced subgraph is denoted by $H = G[A]$.

Let $K_{n}$ denote the \emph{complete graph} on $n$ vertices, where $V(K_{n}) = [n]$ and $E(K_{n}) =  \left\{ij\colon i,j\in V(K_{n}),~i\neq j\right\}$. 
Then, for any $S\subseteq E(K_{n})$ and $x\in\mathbb{R}^{\abs{E(K_{n})}}$, define
\[
x(S) := \sum_{ij\in S}x_{ij}. 
\]
Given $n\geq 3$ and $a\in\mathbb{R}^{\abs{E(K_{n})}}$, the \emph{symmetric traveling salesman problem}, denoted $\tsp{a}$, is defined as follows:
\begin{mini!}
	{}{\sum_{ij\in E(K_{n})}a_{ij}x_{ij}}{}{}\label{eq:tsp-obj}
	\addConstraint{\sum_{j\neq i\colon j\in V(K_{n})}x_{ij}=2,\quad\forall i\in V(K_{n})}
	\addConstraint{x\left(E(K_{n}[A])\right)\leq\abs{A} - 1,\quad\forall\emptyset\subset A\subset V(K_{n})}
	\addConstraint{x_{ij}\in\{0,1\},\quad\forall ij\in E(K_{n}).}
\end{mini!}
It is well-known that every feasible solution of $\tsp{a}$ corresponds to a unique \emph{tour} $T\subset K_{n}$, i.e., a simple cycle of length $n$, where $\bar{x}_{ij} = 1$ if and only if $ij\in E(T)$, for all $ij\in E(K_{n})$.
We say that $T$ is an \emph{optimal tour} if it corresponds to an optimal solution of $\tsp{a}$. 

For $\epsilon>0$, the optimal diameter binary program for $\tsp{a}$, denoted $\conj{\textrm{TSD}}$, is defined by
\begin{maxi!}
	{}{\sum_{ij\in E(K_{n})}-a_{ij}\left(x_{ij} + y_{ij}\right) - \epsilon z_{ij}}{}{}\label{eq:dtsp-obj}
	\addConstraint{\sum_{j\neq i\colon j\in V(K_{n})}x_{ij}=2,\quad\forall i\in V(K_{n})}\label{eq:dtsp-constx1}
	\addConstraint{\sum_{j\neq i\colon j\in V(K_{n})}y_{ij}=2,\quad\forall i\in V(K_{n})}\label{eq:dtsp-consty1}
	\addConstraint{x\left(E(K_{n}[A])\right)\leq\abs{A} - 1,\quad\forall\emptyset\subset A\subset V(K_{n})}\label{eq:dtsp-constx2}
	\addConstraint{y\left(E(K_{n}[A])\right)\leq\abs{A} - 1,\quad\forall\emptyset\subset A\subset V(K_{n})}\label{eq:dtsp-consty2}
	\addConstraint{x_{ij}+y_{ij}-z_{ij}\leq 1,\quad\forall ij\in E(K_{n})}\label{eq:dtsp-constz}
	\addConstraint{x_{ij},y_{ij},z_{ij}\in\{0,1\},\quad\forall ij\in E(K_{n}).}
\end{maxi!}
For two tours $T_{1},T_{2}\subset K_{n}$, define the set of \emph{discordant edges} by
\[
D\left(T_{1},T_{2}\right) := \left\{ij\in E(K_{n})\colon \left(ij\in T_{1}\land ij\notin T_{2}\right)\lor\left(ij\in T_{2}\land ij\notin T_{1}\right)\right\}.
\]
The following result shows that the optimal diameter of $\tsp{a}$ is proportional to the max of $\abs{D\left(T_{1},T_{2}\right)}$ over all optimal tours of $\tsp{a}$.
\begin{proposition}\label{prop:tsp-obj}
Let $S$ denote the set of optimal tours of $\tsp{a}$. 
Then, the optimal diameter of $\tsp{a}$ satisfies
\[
\dia{\tsp{a}} = 2\cdot\argmax_{T_{1},T_{2}\in S}\abs{D\left(T_{1},T_{2}\right)}.
\]
\end{proposition}
\begin{proof}
Note that $\norm{\bar{x}}^{2} = n$ for all $\bar{x}\in\fes{\tsp{a}}$. 
By Corollary~\ref{cor:bar-bpd-obj}, there exists an $\epsilon>0$ such that every $x^{*}\oplus y^{*}\oplus z^{*}\in\opt{\conj{\textrm{TSD}}}$ satisfies 
\[
\dia{\tsp{a}} = 2\left(n - e^{T}z^{*}\right).
\]
Recall from the proof of Theorem~\ref{thm:bpd-obj} that $x^{*},y^{*}\in\opt{\tsp{a}}$ such that $e^{T}z^{*}$ is minimized. 

Now, consider the optimal tours $T_{1},T_{2}\subset K_{n}$ corresponding to $x^{*}$ and $y^{*}$, respectively. 
Then, by Proposition~\ref{prop:bpd-obj}, $z^{*}_{ij}=1$ if and only if $ij\in T_{1}$ and $ij\in T_{2}$.
Therefore,
\[
n - e^{T}z^{*} = \abs{D\left(T_{1},T_{2}\right)}.
\]
The result follows from notating that $\abs{D\left(T_{1},T_{2}\right)}$ is maximized since $e^{T}z^{*}$ is minimized. 
\end{proof}

From the proof of Proposition~\ref{prop:tsp-obj}, there exists an $\epsilon>0$ such that every $x^{*}\oplus y^{*}\oplus z^{*}\in\opt{\conj{\textrm{TSD}}}$ satisfies
\begin{equation}\label{eq:tsp-kt-obj}
\argmax_{T_{1},T_{2}\in S}\abs{D\left(T_{1},T_{2}\right)} = n - e^{T}z^{*},
\end{equation}
where $S$ is the set of all optimal tours of $\tsp{a}$. 
Furthermore, by Corollary~\ref{cor:int-bpd-obj} and~\ref{cor:rat-bpd-obj}, if the entries of $a_{ij}$, for $ij\in E(K_{n})$ are integer or rational valued, then there is a readily computable value of $\epsilon$ that will suffice for~\eqref{eq:tsp-kt-obj} to hold. 

Now, we turn our attention to the diameter polytope of the TSP, which is defined as follows
\[
P^{n}_{\conj{\textrm{TSD}}} := \conv\left\{x\oplus y\oplus z\in\{0,1\}^{3\binom{n}{2}}\colon\text{constraints~\eqref{eq:dtsp-constx1}--~\eqref{eq:dtsp-constz} hold}\right\}.
\]
The following result shows that the conditions of Corollary~\ref{cor:dia-poly-dim2} apply to $\tsp{a}$.
\begin{lemma}\label{lem:tsp-fes}
Let $n\geq 5$.
Then, for any $\bar{x}\in\fes{\tsp{a}}$, there exists a $\bar{y}\in\fes{\tsp{a}}$ such that $\bar{x}+\bar{y}\leq e$. 
\end{lemma}
\begin{proof}
Let $\bar{x}\in\fes{\tsp{a}}$ and $T_{1}\subset K_{n}$ denote the corresponding tour. 
If $n=5$, then the complement $\conj{T}_{1}$ is a tour in $K_{n}$ such that $\abs{D(T_{1},\conj{T}_{1})} = n$. 
Let $\bar{y}\in\fes{\tsp{a}}$ denote the feasible solution corresponding to $\conj{T}_{1}$.
Then, $\bar{x}+\bar{y}\leq e$ and the result follows.

Suppose that $n>5$ and define the graph $G := \left(V(K_{n}),E(K_{n})\setminus{E(T_{1})}\right)$.
Then, the degree of each vertex $v\in V(G)$ satisfies
\[
\deg{v} \geq n - 3 \geq \frac{n}{2}.
\]
Therefore, by~\cite[Theorem 3]{Dirac1952}, it follows that $G$ is Hamiltonian, i.e., there exists a tour $T_{2}$ in $G$.
Let $\bar{y}\in\fes{\tsp{a}}$ denote the feasible solution corresponding to $T_{2}$.
Then, $\bar{x}+\bar{y}\leq e$ and the result follows. 
\end{proof}
Hence, we have the following result.
\begin{proposition}\label{prop:tsp-dia-poly-dim}
Let $n\geq 5$.
Then, the dimension of $P^{n}_{\conj{\textrm{TSD}}}$ satisfies
\[
\dim{P^{n}_{\conj{\textrm{TSD}}}} = \frac{3n^{2}-7n}{2}
\]
\end{proposition}
\begin{proof}
Since $n\geq 5$, Lemma~\ref{lem:tsp-fes} implies that the conditions of Corollary~\ref{cor:dia-poly-dim2} apply to $\tsp{a}$.
Let $Mx=d$, where $M\in\mathbb{R}^{m\times n(n-1)}$ and $d\in\mathbb{R}^{m}$, be a minimal equation of $Q^{n}_{T}$, i.e., the ($n$-city) symmetric traveling salesman polytope. 
Then, Corollary~\ref{cor:dia-poly-dim2} implies that
\[
\dim{P^{n}_{\conj{\textrm{TSD}}}} = 3\binom{n}{2} - 2\rank{M},
\]
where $\rank{M} = n$ by~\cite[Proposition 0]{Maurras1975}.
\end{proof}

Note that, for $n=4$, the dimension equation in Proposition~\ref{prop:tsp-dia-poly-dim} was verified using {\tt polymake}~\cite{polymake2000}.
Also, for $n\geq 5$, Lemma~\ref{lem:tsp-fes} implies that the conditions in Theorems~\ref{thm:dia-poly-facet} --~\ref{thm:dia-poly-ktfacet} apply to $\tsp{a}$, and the result below follows immediately. 
In particular, note that Proposition~\ref{prop:tsp-dia-poly-facets}(a) can be used in conjunction with~\cite[Theorem 14]{Grotschel1985:3} to produce multiple facet inequalities for $P^{n}_{\conj{\textrm{TSD}}}$.
\begin{proposition}\label{prop:tsp-dia-poly-facets}
Let $n\geq 5$ and suppose that $a^{T}x\leq a_{0}$ defines a facet of $Q^{n}_{T}$.
Then, the following are facet inequalities of $P^{n}_{\conj{\textrm{TSD}}}$:
\begin{enumerate}[(a)]
\item $\hat{a}^{T}\left(x\oplus y\oplus z\right)\leq a_{0}$, for $\hat{a}=a\oplus 0\oplus 0$ and $\hat{a}=0\oplus a\oplus 0$,
\item $z_{ij}\geq 0$ and $z_{ij}\leq 1$, for all $i,j\in E(K_{n})$,
\item $x_{ij}+y_{ij}-z_{ij}\leq 1$, for all $i,j\in E(K_{n})$.
\end{enumerate}
\end{proposition} 
\section{Conclusion}
The diameter binary program is a novel tool for analyzing the set of optimal solutions for a given feasible binary program. 
In particular, the optima of the diameter binary program contains two optimal solutions of the given feasible binary program that are as diverse as possible with respect to the optimal diameter.

Under suitable conditions, the dimension of the polytope (Corollary~\ref{cor:dia-poly-dim1} and~\ref{cor:dia-poly-dim2}) and certain facet inequalities (Theorem~\ref{thm:dia-poly-facet}) can be found from the dimension and facets, respectively, of the underlying polytope of the given binary program. 
In addition, the trivial hypercube constraints $0\leq z_{i}\leq 1$ define facets (Theorem~\ref{thm:dia-poly-zfacet}), for all $i\in [n]$, as do the constraints in~\eqref{eq:bpd-constz1} (Theorem~\ref{thm:dia-poly-ktfacet}).

These suitable conditions apply to many famous binary programs, such as those corresponding to the linear ordering problem and the symmetric traveling salesman problem. 
When considering these binary programs, the additional problem-specific structure reveals other interesting facets.
For instance, the diameter polytope of the linear ordering problem has a lifting result (Theorem~\ref{thm:lop-dia-poly-facet-lift}), which shows that all facets in a lower dimension can be ``lifted" to facets in a higher dimension.
When $n=2$ and $n=3$, {\tt polymake} reveals that all facets of the linear ordering problem are described by Proposition~\ref{prop:lop-dia-poly-facets}.
For $n=4$, there are $483,840$ points, which proved to be too many for us to have {\tt polymake} compute the facet inequalities in a reasonable amount of time. 

For the diameter polytope of the symmetric traveling salesman problem, when $n=4$, {\tt polymake} reveals the following additional facets:
\begin{align*}
x_{12} + x_{13} + y_{12} + y_{24} + z_{23} &\geq 3, \\
x_{12} + x_{13} + y_{12} + y_{24} + z_{14} &\geq 3,
\end{align*}
which are not described by Proposition~\ref{prop:tsp-dia-poly-facets}.
These facets seem to occur from the interaction between the two optimal solutions of the TSP. 
When $n=5$, {\tt polymake} reveals that the majority of the facets are of a similar form.
Note that the python code used to generate these facet inequalities is available at \url{https://github.com/trcameron/Diameter-Polytopes}.
In addition, we have included a python implementation of the diameter linear ordering program, which uses {\tt CPLEX}~\cite{cplex2019} as the underlying optimization solver. 

Future research includes the investigation of these additional facets for the diameter polytope of the symmetric traveling salesman problem.
In addition, we are interested in applying the diameter binary program to other binary programs, such as those corresponding to the acyclic subgraph problem, the set cover problem, and the knapsack problem.  
Finally, possible generalizations include developing the diameter binary program to allow for the computation of more than two optimal solutions that are as diverse as possible, and perhaps are associated with separate objective functions. 
As noted in the introduction, it is straightforward to change the constraints so that these optima are as uniform as possible.

\bibliographystyle{plain}
  \bibliography{Diameter-BP}

\begin{thebibliography}{10}

\bibitem{Dirac1952}
G.~A. Dirac.
\newblock Some theorems on abstract graphs.
\newblock {\em Proc. London Math. Soc.}, 3(1):69--81, 1952.

\bibitem{polymake2000}
Ewgenij Gawrilow and Michael Joswig.
\newblock {\tt polymake}: a framework for analyzing convex polytopes.
\newblock In {\em Polytopes---combinatorics and computation ({O}berwolfach,
  1997)}, volume~29 of {\em DMV Sem.}, pages 43--73. Birkh\"auser, Basel, 2000.

\bibitem{Glover1998}
F.~Glover, C.~C. Kuo, and K.~S. Dhir.
\newblock Heuristic algorithms for the maximum diversity problem.
\newblock {\em J. Inform. Optim. Sci.}, 19(1):109--132, 1998.

\bibitem{Grotschel1985:2}
M.~Gr{\"o}tschel, M.~J{\"u}nger, and G.~Reinelt.
\newblock Facets of the linear ordering problem.
\newblock {\em Mathematical Programming}, 33:43--60, 1985.

\bibitem{Grotschel1993}
M.~Gr{\"o}tschel, L.~Lovasz, and A.~Schrijver.
\newblock {\em Geometric Algorithms and Combinatorial Optimization}.
\newblock Springer-Verlag, Berlin, Germany, 1993.

\bibitem{Grotschel1985:3}
M.~Gr{\"o}tschel and M.~W. Padberg.
\newblock Polyhedral theory.
\newblock In {\em The {T}raveling salesman problem}, chapter~8, pages 251--302.
  John Wiley \& Sons, Philadelphia, PA, 1985.

\bibitem{cplex2019}
IBM.
\newblock {\em IBM ILOG CPLEX 12.9 User’s Manual}.
\newblock IBM ILOG CPLEX Division, Incline Village, NV, 2019.

\bibitem{Kondo2014}
Y.~Kondo.
\newblock Triangulation of input–output tables based on mixed integer
  programs for inter-temporal and inter-regional comparison of production
  structures.
\newblock {\em J. Econ. Struct.}, 3(2):1--19, 2014.

\bibitem{Kuo1993}
C.~.C. Kuo.
\newblock Analyzing and modeling the maximum diversity problem by zero-one
  programming.
\newblock {\em Decision Sci.}, 24(6):1171--1185, 1993.

\bibitem{Lee2004}
J.~Lee.
\newblock {\em A First Course in Combinatorial Optimization}.
\newblock Cambridge University Press, Cambridge, England, 2004.

\bibitem{Marti2011}
R.~Mart{\'i} and G.~Reinelt.
\newblock {\em The {L}inear {O}rdering {P}roblem}.
\newblock Springer-Verlag, Berlin, Germany, 2011.

\bibitem{Maurras1975}
J.~F. Maurras.
\newblock Some results on the convex hull of the hamiltonian cycles of symetric
  complete graphs.
\newblock In {\em Combinatorial Programming: Methods and Applications}, pages
  179--190, Dordrecht, Netherlands, 1975. Springer.

\bibitem{Petit2019}
T.~Petit and A.~C. Trapp.
\newblock Enriching solutions to combinatorial problems via solution
  engineering.
\newblock {\em INFORMS J. Comput.}, 31(3):429--444, 2019.

\bibitem{Tsai2008}
J.-F. Tsai, M.-H. Lin, and Y.-C. Hu.
\newblock Finding multiple solutions to general integer linear programs.
\newblock {\em European Journal of Operational Research}, 184:802--809, 2008.

\end{thebibliography}

\end{document}